\documentclass[11pt,a4paper]{article}

\usepackage{amssymb}
\usepackage{tikz}
\usepackage[textures]{epsfig}
\usepackage{amscd}
\usepackage{mathrsfs}
\usepackage{xcolor}
\usepackage{cancel}
\usepackage[centertags]{amsmath}
\usepackage{amsfonts}
\usepackage{amssymb}
\usepackage{amsthm}
\usepackage{epsfig}
\usepackage{graphicx}
\usepackage{amsthm}
\usepackage{amssymb,latexsym}
\usepackage{amsfonts,amsmath}
\usepackage[T1]{fontenc}
\usepackage{amsmath,amssymb}
\usepackage{latexsym}
\usepackage{amssymb}
\usepackage{amsfonts}
\usepackage{fancyhdr}
 \usepackage{mathtools}
 \mathtoolsset{showonlyrefs}

\def\ni{\noindent}
\def\beq{\arraycolsep1pt\begin{eqnarray*}}
\def\eeq{\end{eqnarray*}}

\newcommand{\HH}{{\cal H}}

\newcommand{\SH}{{\mathscr{H}}}

\newcommand{\R}{{\mathbb R}}

\newcommand\ov{\overline}

\newtheorem{Thm}{Theorem}[section]
\newtheorem{lem}[Thm]{Lemma}
\newtheorem{pro}[Thm]{Proposition}

\newtheorem{re}[Thm]{Remark}

\newtheorem{cor}[Thm]{Corollary}

\def\bea{\begin{eqnarray}}
\def\eea{\end{eqnarray}}
\numberwithin{equation}{section}

\date{}

\title{Landesman-Lazer conditions for systems involving twist and positively homogeneous Hamiltonian  systems}

\author{Natnael Gezahegn Mamo and Wahid  Ullah}

\begin{document}

\maketitle

\begin{abstract}
We present multiplicity results for the periodic and Neumann-type boundary value problems associated with coupled Hamiltonian systems. For the periodic problem, we couple a system having twist condition with another one  whose nonlinearity lies between the gradients of two positive and positively $2$-homogeneous Hamiltonain functions. Concerning the Neumann-type problem, we treat the same system without any twist assumption. We examine the cases of nonresonance, simple resonance, and double resonance by imposing some kind of Landesman--Lazer conditions. 
\end{abstract}

\section{Introduction }
The main motivation of this paper is arising from the two open problems given under a recent paper \cite[Section 5, Problems 2 and 4]{FonSfeToa2024Pre}.
To this aim, we first consider the periodic problem associated with a  system of the form
\begin{equation}\label{sys:original}
\begin{cases} 
\dot{x}= \nabla_y \mathcal{H}(t,x,y)+\nabla_y P(t,x,y,w)\,,\\  
\dot{y}= -\nabla_x \mathcal{H}(t,x,y)-\nabla_xP(t,x,y,w)\,, \\
J\dot{w} = F(t,w)+ \,\nabla_{w} P(t,x,y,w)\,,
\end{cases}
\end{equation}
where $F(t,w)$ is the gradient of a Hamiltonian function $K: \R \times \R^{2} \to \R$, i.e.,
$$
F(t,w) = \nabla_{w} K(t,w)\,,
$$
and $\mathcal{H}: \R \times \R^{2M} \to \R$. In addition, all the involved functions in~\eqref{sys:original} are continuous,  $T$-periodic in the variable $t$ and continuously differentiable with respect to $(x,y,w)$. We denote by $J=\begin{pmatrix}
    0 & -1\\
    1 & 0
\end{pmatrix}$ the standard symplectic matrix. 

\medbreak

To be more precise, the idea is to consider a Poincar\'e--Birkhoff situation for system
\begin{equation}\label{sys:twist}
\dot{x}= \nabla_y H(t,x,y)\,,\qquad  \dot{y}= -\nabla_x H(t,x,y)\,,
\end{equation}
while for system 
\begin{equation}\label{sys:phh}
J\dot{w} = F(t,w)\,,
\end{equation}
we assume that $F(t,w)$ is the combination of gradients of two positive and positively $2$-homogeneous Hamiltonian functions together with a perturbation term.

\medbreak

We now explain some basic preliminary concepts about positively homogeneous Hamiltonian systems. We refer the reader to~\cite{Fon2004JDE, FonGar2011JDE, FonUll2024DIE} for the detailed explanation. If $\mathscr{H} :\R^2 \to \R$ is a $C^1$-function satisfying 
\begin{equation}\label{homogenity}
    0< \mathscr{H}(\lambda w)= \lambda^2 \mathscr{H}( w),\quad \mbox{ for every }w\in\R^2 \backslash \{0\}\,\mbox{ and } \lambda >0\,,
\end{equation}
then $\mathscr{H}$ is said to be positively homogeneous of degree 2
and positive. In this setting, the origin is an isochronous center for the autonomous Hamiltonian system
\begin{equation}\label{system_autonomous}
J\dot{w} = \nabla \mathscr{H}(w),\quad \mbox{ for every } \, w \in \R^2\,,
\end{equation}
which means that aside from the origin, which is an equilibrium, all solutions of system~\eqref{system_autonomous} are periodic and have the same minimal period.

\medbreak 
 Notice that if $\mathscr{H}$ is positively homogeneous of degree $2$, then the Euler identity holds:
\begin{equation}\label{Euler_idty}
   \big\langle \nabla\mathscr{H}(w)\,,w
\big\rangle = 2\mathscr{H}(w),\quad\mbox{ for every }w\in \R^2\,. 
\end{equation}

\medbreak

In \cite{FonUre2017AIHPANL}, a higher dimensional version of the Poincar\'e--Birkhoff Theorem was studied for a periodic problem associated with the Hamiltonian system~\eqref{sys:twist}. Recently, several authors have discussed about the existence of multiple solutions for a coupled systems, where the coupling is between system~\eqref{sys:twist} with some other systems involving different types of structures, see for example~\cite{CheQia2022JDE, FonGarSfe2023JMAA, FonGid2020NODEA, FonUll2024JDE}. In~\cite{FonUll2024DIE, FonUll2024NODEA}, system~\eqref{sys:twist} is coupled with some isochronous centre by using a nonresonance condition. In~\cite{FonMamSfe2024Preprint}, the coupling was with a scalar second order equation, where the nonlinearity has an asymmetric behaviour combined with some Landesman--Lazer conditions. In that paper, nonresonance, simple resonance and double resonance situations were considered by implementing some kind of Landesman--Lazer conditions.

\medbreak

Let us explain what the Landesman--Lazer condition is. This condition has been introduced by Lazer and Leach~\cite{LazLea1969AMPA} for the periodic problem associated with a scalar second order ordinary differential equations of the form
\begin{equation}\label{Scalar ODE}
\ddot{u}+\sigma u +h(t,u)=0\,,
\end{equation}
where $h$ is continuous, uniformly bounded and $T$-periodic in $t$ with $\sigma=\big(\frac{2\pi n}{T}\big)^2$ for some positive integer $n$. In this setting they proved that a sufficient condition for the existence of a $T$-periodic solution is the following: for every non-zero $\eta$ satisfying $\ddot{\eta}+\sigma \eta=0$, we have 
\begin{equation}\label{eq:of:LL}
    \int_{\{\eta < 0\}}\limsup\limits_{u\to -\infty}h(t,u)\eta(t) dt + \int_{\{\eta > 0\}}\limsup\limits_{u\to +\infty}h(t,u)\eta(t) dt >0\,.
\end{equation} 
Notice that when $h(t,u)$ is increasing in $u$ this condition becomes a necessary and sufficient condition for the existence of a solution of~\eqref{Scalar ODE}. One year later, Landesman and Lazer~\cite{LanLaz1970JMM} introduced a similar condition for a Dirichlet problem associated with an elliptic operator. Since then,~\eqref{eq:of:LL} is referred to as Landesman--Lazer condition. This condition is crucial for  the nonlinearity to be kept sufficiently far from resonance. A lot of generalizations have been made by taking their work as a steeping stone, see for example~\cite{BreNir1978ASNSP, Dra1990JDE, FabFon1992NA, Maw1977CSMUB, Sfe2020TMNA}.

\medbreak 

The paper is organized as follows. In Section~\ref{main_result} we state our main results for the periodic problem concerning nonresonance, simple resonance and double resonance situations. The proofs of these results are provided in Section~\ref{Proofs}. Section~\ref{variants} is devoted to some variants of our results. The application will be given through some examples in Section~\ref{examples}. In Section~\ref{NBVP} we state and prove an analogue result concerning a Neumann-type boundary value problem associated with system \eqref{sys:original}, and finally we state some open problems in Section \ref{open}.

\section{Main results for the periodic problem}\label{main_result}
We first list all our main assumptions and then the corresponding results related to system~\eqref{sys:original}, where all the involved functions are continuous, $T$-periodic in the variable $t$ and continuously differentiable with respect to $(x,y,w)$. The proofs will be given in the next section.
\medbreak
Here are  our main assumptions. The first one is concerning the periodicity for the Hamiltonian function $\HH$. We write 
\begin{align*}
&x=(x_1 , \dots , x_M) \in\R^{M}, \quad y=(y_1 , \dots , y_M)\in\R^{M}.
\end{align*}

\medbreak

\noindent$A1.$ The function $\HH(t,x,y)$ is $2\pi$-periodic in each $x_i$ for $i=1, \dots, M$.

\medbreak

The following assumption is on the  periodicity and boundedness condition for the coupling function $P$.

\medbreak 

\noindent $A2.$ The function $P(t,x,y,w)$ is $2\pi$-periodic in each $x_i$ and has a bounded gradient with respect to $(x,y,w)$.
In particular, there exists a constant $\ov m$ such that 
$$
| \nabla_ w P(t,x,y,w)| \leq \ov m\,,
\qquad \text{for every $(t,x,y,w)$}\,.
$$

\medbreak

In order to introduce a { \em twist assumption}, adapted to our setting (cf.~\cite{FonGarSfe2023JMAA, FonUll2024DIE, FonUll2024NODEA}), we first consider the case when $\mathcal{D}$ is a rectangle in $\mathbb{R}^M$, i.e.,
$$
\mathcal{D}= [a_1, b_1] \times \cdots \times [a_M,b_M]\,,
$$
and we denote by $\mathring{\mathcal D}$ its interior.

\medbreak

\noindent $A3.$  There exists an $M$-tuple $\sigma=(\sigma_1 , \dots , \sigma_M) \in \{-1,1\}^{M}$ such that for every $C^1$-function $\mathcal{W}:[0,T] \to \R^{2}$, all the solutions $(x,y)$ of system
\begin{equation}\label{eq:of:hamiltonian:unperterbed:for:higher}
\begin{cases}
    \dot{x}= \nabla_{y} \mathcal{H}(t,x,y)+\nabla_{y}P(t,x,y,\mathcal{W}(t))\,,\\
    \dot{y}= -\nabla_{x} \mathcal{H}(t,x,y)-\nabla_{x}P(t,x,y,\mathcal{W}(t)) \,,   
\end{cases} 
\end{equation}
starting with $y(0)\in{\cal D}$ are defined on $[0,T]$, and for every $i=1, \dots, M$, we have
$$
\begin{cases}
y_{i}(0)=a_i \quad\Rightarrow\quad \sigma_{i}[x_{i}(T)-x_{i}(0)] < 0\,,\\
y_{i}(0)=b_i\quad\Rightarrow\quad \sigma_{i}[x_{i}(T)-x_{i}(0)] > 0\,.
\end{cases}
$$

\medbreak

We now introduce a structural assumption for the function $F$.

\medbreak

\noindent $A4.$ There are two functions $\gamma: \R\times \R^2 \to [0,1]$ and $Q: \R\times \R^2 \to \R$ such that
\begin{equation}\label{covex_comb.}
F(t,w)=(1-\gamma(t,w))\nabla \mathscr{H}_1 (w)+\gamma(t,w)\nabla \mathscr{H}_2 (w) + \nabla_w Q(t,w)\,,
\end{equation}
where all the functions are continuous, $T$-periodic in $t$, and $\nabla_w Q(t,w)$ is uniformly bounded, i.e., there exists $\widetilde{C}>0$ such that
\begin{equation}\label{eq:of:bdd:of:q}
|\nabla_w Q(t,w)| \leq  \widetilde{C}\,, \quad \hbox{ for every }\, (t,w) \in [0,T] \times \R^2\,.
\end{equation}

\medbreak

It is our aim to consider the case when $F(t,w)$ in some sense interacts asymptotically with the gradients of the two positive and positively $2$-homogeneous Hamiltonian functions $\mathscr{H}_1$ and $\mathscr{H}_2$ which satisfy~\eqref{homogenity} and 
\begin{equation}\label{H1_leq_H2}
    \mathscr{H}_1(w) \leq \mathscr{H}_2(w)\,,\quad \mbox{for every } w\in \R^2\,. 
\end{equation}

We notice that the function $F(t, \cdot)$ has at most linear growth, i.e., there exists $C>0$ such that
\begin{equation}\label{eq:of:lin:growth}
    |F(t,w)| \leq C (1+ |w|) \quad \hbox{ for every } w \in \R^2\,.
\end{equation}

We fix $\varphi$ and $\psi$ such that
$$
J \dot{\varphi}= \nabla \SH_1(\varphi), \quad J \dot{\psi}= \nabla \SH_2(\psi)\,,
$$
and
\begin{equation}\label{half}
    \mathscr{H}_1 (\varphi (t))= \mathscr{H}_2 (\psi (t))=\frac{1}{2}\,,\quad \mbox{ for every }t\in [0,T].
\end{equation}
The choice is not restrictive.
\medbreak
Notice that once we fixed $\varphi$ and $\psi$, all the non-zero solutions of
$$
J\dot{w}=\nabla \mathscr{H}_1(w)\quad \mbox{and} \quad J\dot{ \widetilde{w}}=\nabla \mathscr{H}_2( \widetilde{w})\,,
$$
respectively, are of the form $ w(t)= R\, \varphi(t+s)$ for all $s \in [0, \tau_1[ $ and $\widetilde{w}(t)=  \widetilde{R}\, \psi(t+s) $ for all $s \in [0, \tau_2[ $ with some positive constants $R$ and $\widetilde{R}$, where $\tau_1$ and $\tau_2$ represent their minimal periods, respectively. 

\medbreak

Here is our first main result which is associated with a nonresonant situation.
\begin{Thm}\label{Thm_nonres}
    Let $A1-A4$ hold true, and assume that there exists a positive integer $N$ such that
    \begin{equation}\label{nonres_condition}
\frac{T}{N+1}< \tau_2 \leq \tau_1 < \frac{T}{N}\,.
    \end{equation}
    Then system~\eqref{sys:original} has at least $M+1$ geometrically distinct $T$-periodic solutions, with $y(0)\in \mathring{\mathcal{D}}$.
\end{Thm}

Due to the $2\pi$-periodicity assumption of $\mathcal{H}$ in the variable $x_i$, if $(x(t),y(t))$ is a $T$-periodic  solution of system~\eqref{sys:twist}, then we can find infinitely many others by adding an integer multiple of $2\pi$ to $x_i(t)$. If one is not obtained from the other in this way, then we call the two $T$-periodic solutions  geometrically distinct.

\medbreak

We now consider the situation where resonance can occur from below and for this we state the following assumption.

\medbreak

\noindent $A5.$ For every $\theta  \in[0,T]$, one has
\begin{multline}\label{LL_1a}
\int_{0}^{T}\liminf\limits_{(\lambda,\,s)\to(+\infty,\,\theta)}\big[\big\langle F\big(t,\lambda \varphi(t+s)\big),\,\varphi(t+s) \big\rangle -2\lambda \mathscr{H}_1\big(\varphi(t)\big)\big] \,dt\\
>\ov{m}\int_{0}^{T}|\varphi(t)| dt\,.
\end{multline}

\begin{Thm}\label{Thm_simple_res:below}
    Let $A1-A5$ hold true, and assume that there exists a positive integer $N$ such that
    \begin{equation}\label{nonres_condition:below}
\frac{T}{N+1}< \tau_2 \leq \tau_1 = \frac{T}{N}\,.
    \end{equation}
  Then the same conclusion of Theorem~\ref{Thm_nonres} holds.
\end{Thm}

\medbreak

Symmetrically, we can discuss the situation of simple resonance from above, and for this we state the following assumption.

\medbreak

\noindent $A6.$
For every $\theta  \in[0,T]$, one has
\begin{multline}\label{LL_1b}
\int_{0}^{T}\liminf\limits_{(\lambda,\,s)\to(+\infty,\,\theta)}\big[2\lambda \mathscr{H}_2\big(\psi(t)\big)-\big\langle F\big(t,\lambda \psi(t+s)\big),\,\psi(t+s) \big\rangle \big] \,dt\\
>\ov{m}\int_{0}^{T}|\psi(t)| dt\,.
\end{multline}

\begin{Thm}\label{Thm_simple_res:above}
    Let $A1-A4$, $A6$ hold true, and assume that there exists a positive integer $N$ such that
    \begin{equation}\label{nonres_condition:above}
\frac{T}{N+1}= \tau_2 \leq \tau_1 < \frac{T}{N}\,.
    \end{equation}
   Then the same conclusion of Theorem~\ref{Thm_nonres} holds.
\end{Thm}

\medbreak
Concerning the double resonance case, we have the following result.
\begin{Thm}\label{Thm_res}
    Let $A1-A6$ hold true, and $\SH_1(w) < \SH_2(w)$ for all $w \in \R^2 \setminus \{0\}$. Assume moreover that there exists a positive integer $N$ such that
    \begin{equation}\label{res_condition}
\tau_2=\frac{T}{N+1} \quad \hbox{ and } \quad \tau_1 = \frac{T}{N}\,.
    \end{equation}  
    Then the same conclusion of Theorem~\ref{Thm_nonres} holds.
\end{Thm}

\section{Proofs of the main results}\label{Proofs}

We first modify the original system, adapted to our setting (cf.~\cite{FonSfeToa2024Pre}) as follows. From $A4$, and the fact that $F(t,w)$ is the gradient of a Hamiltonian function $K$, we can find a function $\varPhi: \R \times \R^{2} \to \R$ such that $F(t,w)=\nabla_w \varPhi (t,w)+\nabla_w Q(t,w)$, and
\begin{equation}\label{nabla_phi}
\nabla_w \varPhi (t,w)=(1-\gamma(t,w))\nabla \mathscr{H}_1 (w)+\gamma(t,w)\nabla \mathscr{H}_2 (w).
\end{equation}
Writing 
$$
\varPhi (t,w)=\varPhi (t,0)+\int_{0}^{1} \langle \nabla_w \varPhi (t,sw), w\rangle\,ds\,,
$$
we can assume without loss of generality that $\varPhi (t,0)=0$. Thus, using~\eqref{Euler_idty},~\eqref{H1_leq_H2} and the fact that for $i=1,2$, $\nabla \mathscr{H}_i$ is positively homogeneous of order one,  we have 
\begin{equation}\label{phi_between}
    \mathscr{H}_1(w)\leq \varPhi(t,w)\leq \mathscr{H}_2(w)\,.
\end{equation}

\medbreak

For every $\rho >1$, let $\eta_{\rho} :\R \to [0,1]$ be a $C^\infty$-function such that 
\begin{equation*}
    \eta_{\rho}(\xi)=\begin{cases}
        1 \quad \mbox{if } \xi \leq \rho\,,\\
        0 \quad \mbox{if } \xi \geq \rho^3\,,
    \end{cases}
\end{equation*}
and 
\begin{equation}\label{eta_prime}
    -\frac{1}{\xi \ln \xi}\leq \eta_\rho '(\xi)\leq 0\,,\quad \mbox{for every }\xi\geq \rho\,. 
\end{equation}
The existence of such a function is guaranteed by the fact that
\begin{equation*}
    \int_{\rho}^{\rho^{3}}\frac{d\xi}{\xi \ln \xi}>1\,.
    \end{equation*}
We now define the function $\varPhi_\rho :\R\times \R^2 \to \R$ as 
\begin{equation*}
    \varPhi_\rho (t,w)=\begin{cases}
        \varPhi(t,w) \quad &\mbox{ if } |w|\leq \rho \,,\\
        \eta_{\rho}(|w|)\varPhi(t,w)+(1-\eta_{\rho}(|w|)\frac{1}{2}[\mathscr{H}_1 (w)+\mathscr{H}_2 (w)]\quad &\mbox{ if } \rho \leq |w| \leq \rho^{3}\,,\\
        \frac{1}{2}[\mathscr{H}_1 (w)+\mathscr{H}_2 (w)]&\mbox{ if }  |w| \geq \rho^{3}\,,
    \end{cases}
\end{equation*}
and consider the following modified system 
\begin{equation}\label{sys_modified}
\begin{cases} 
\dot{x}= \nabla_y \HH(t,x,y)+\nabla_y P(t,x,y,w)\,,\\  
\dot{y}= -\nabla_x \HH(t,x,y)-\nabla_xP(t,x,y,w)\,, \\
J\dot{w} = F_\rho (t,w)+ \,\nabla_{w} P(t,x,y,w)\,,
\end{cases}
\end{equation}
where 
\begin{equation}\label{F_rho}
    F_\rho (t,w)=\nabla_w \varPhi_\rho (t,w) + \nabla_w Q(t,w)\,.
\end{equation} 
Notice that $\nabla_w \varPhi_\rho (t,w)$ can be decomposed as 
$$\nabla_w \varPhi_\rho (t,w)=(1-\gamma_\rho (t,w))\nabla H_1 (w)+\gamma_\rho (t,w)\nabla H_2 (w) + v_\rho (t,w)\,,$$
where $\gamma_\rho: \R\times \R^2 \to [0,1]$ is defined by
\begin{equation*}
    \gamma_\rho (t,w)=\begin{cases}
        \gamma (t,w) &\quad \mbox{ if }|w|\leq \rho\,,\\
        \eta_\rho (|w|)\gamma (t,w) +\frac{1}{2}(1-\eta_\rho (|w|)) &\quad \mbox{ if }\rho\leq |w|\leq \rho^3\,,\\
        \frac{1}{2} &\quad \mbox{ if }|w|\geq \rho^3\,,
    \end{cases}
\end{equation*}
and $v_\rho :\R \times \R^2 \to \R^2$ by
\begin{equation*}
    v_\rho (t,w)=\begin{cases}
        0 &\quad \mbox{ if }|w|\leq \rho\,,\\
        \eta_\rho ' (|w|)|w|^{-1} [\varPhi(t,w)-\frac{1}{2}(\mathscr{H}_1(w)+\mathscr{H}_2(w))]w
        &\quad \mbox{ if }\rho\leq |w|\leq \rho^3\,,\\
        0 &\quad \mbox{ if }|w|\geq \rho^3\,.
    \end{cases}
\end{equation*}
By~\eqref{homogenity}, for $i=1,2$ there exist $C_i>0$ such that $0< \SH_{i}(w) \leq C_i |w|^2$. So, if $\rho \leq |w| \leq \rho^ 3$, then we have
\begin{align}\label{eq:bdd:of:vrho} \nonumber
    |v_\rho (t,w)|&\leq \frac{1}{|w|\ln(|w|)}|\mathscr{H}_2 (w)-\mathscr{H}_1 (w)|\\ 
    &\leq \frac{C_1 +C_2}{\ln(|w|)}|w| \leq \frac{C_3}{2\ln \rho} |w|\,,
\end{align}
where $C_3 = C_1+C_2$.

\subsection{Proof of Theorem~\ref{Thm_nonres}}\label{Proof_1}
We first provide some a priori bounds so to ensure that the $T$-periodic solutions of the modified system are indeed solutions of the original one.
\begin{pro}\label{pro:for:priori}
There exists $\ov{\rho}>1$ such that, for any $\rho>\ov{\rho}$, every $T$-periodic solution of~\eqref{sys_modified} satisfies $\|w\|_\infty \leq \ov{\rho}$.
\end{pro}
\begin{proof}
    Assume by contradiction that there is a sequence $(\rho_n)_{n}$ in $]1,+\infty[$ and a sequence of $T$-periodic solutions $(x_n, y_n, w_n) $ of~\eqref{sys_modified}, with $\rho=\rho_n$, such that $\rho_n\to +\infty$ and $\|w_n\|_\infty > n$. Let $z_n=\frac{w_n}{\|w_n\|_\infty}$. Then, $z_n$ is $T$-periodic and satisfies 
    \begin{equation}\label{normalized_problem}
    \begin{cases}
        J\dot{z}_n
        =(1-\Gamma_n(t))\nabla\mathscr{H}_1(z_n)+
\Gamma_n(t)\nabla\mathscr{H}_2(z_n)+\kappa_n (t)\,,\\
        z_n (0)= z_n (T)\,,
        
        \end{cases}
    \end{equation}
    where
    $$\Gamma_n(t)=\gamma_{\rho_{n}}(t,\|w_n\|_\infty z_n(t))\,,$$
    and
  \begin{multline}\label{kappa}
      \kappa_n (t)=  \frac{1}{\|w_n\|_\infty}\big[v_{\rho_{n}}(t,\|w_n\|_\infty z_n(t))+
        \nabla_w Q(t,\|w_n\|_\infty z_n(t))+\\ \nabla_w P(t,x_n(t),y_n(t),\|w_n\|_\infty z_n(t))\big]\,.
        \end{multline}
Notice that $\Gamma_n(t)\in [0,1]$ for every $t \in [0,T]$. Moreover, using~\eqref{eq:bdd:of:vrho}, and by the boundedness of $\nabla_w Q$ and $\nabla_w P$, we see that $\kappa_n \to 0$ uniformly in $[0,T]$. The sequence  $(z_n)_{n}$ is thus bounded in $C^1 ([0,T],\R^2)$; therefore there exists a function $z\in C ([0,T],\R^2)$ such that, up to a subsequence, $z_n \to z$ uniformly and  weakly in $H^1$. Since the sequence $(\Gamma_n)_{n}$ is bounded, we can suppose that, up to a subsequence, it converges to some $\Gamma$ weakly in $L^2$ with $\Gamma(t)\in [0,1]$ for almost every $t\in [0,T]$. So, $\|z\|_\infty =1$, and passing to the weak limit in~\eqref{normalized_problem}, $z$ solves
\begin{equation*}
    \begin{cases}
        J\dot{z}
        =(1-\Gamma(t))\nabla\mathscr{H}_1(z)+
\Gamma(t)\nabla\mathscr{H}_2(z)\,,\\
        z (0)= z (T)\,.
        \end{cases}
    \end{equation*}
By~\cite[Lemma 2.3]{FonGar2011JDE}, either $z$ is a solution of $J\dot{z}=\nabla\mathscr{H}_1 (z)$, or it is a solution of $J\dot{z}=\nabla\mathscr{H}_2 (z)$. But this is a contradiction to the fact that $J\dot{z}=\nabla\mathscr{H}_1 (z)$ and $J\dot{z}=\nabla\mathscr{H}_2 (z)$ have only a trivial solution by~\eqref{nonres_condition}. Hence our supposition was wrong, and this completes the proof of the proposition.
\end{proof}

To conclude the proof of the theorem, we fix $\rho \geq \bar\rho$, and rewrite the Hamiltonian function of system~\eqref{sys_modified} as
\begin{equation}\label{eq:of:K:rho}
K_\rho(t,x,y,w)= \HH(t,x,y) + \frac{1}{2}\big[ \mathscr{H}_1(w) + \mathscr{H}_2(w)  \big] + \widetilde{P}_{\rho}(t,x,y,w) \,,
\end{equation}
with 
$$
 \widetilde{P}_{\rho}(t,x,y,w)=P(t,x,y,w) + Q(t,w) + \Phi_{\rho}(t,w) -\frac{1}{2}\big[ \mathscr{H}_1(w) + \mathscr{H}_2(w)  \big]\,. 
$$
Proposition~\ref{pro:for:priori} provides an a priori bound in $C([0,T], \R^2)$ for the $w$ component of the solutions of system~\eqref{sys_modified}. Using the second system in~\eqref{sys_modified} we see that the a priori bound can be extended to the derivative of $w$. This implies that we have an a priori bound in $C^1([0,T], \R^2)$. Hence, by the Ascoli--Arzel\`a Theorem that $w$ belongs to a compact set $W \subseteq C([0,T], \R^2)$. Now since $\varPhi_{\rho}(t,w)- \frac{1}{2}\big[ \mathscr{H}_1(w) + \mathscr{H}_2(w)  \big] = 0$ for $|w| \geq \rho^3$, we have that $\widetilde{P}_{\rho}(t,x,y,w)$ has a bounded gradient with respect to $(x,y,w)$. Thus by using~\cite[Theorem 1.1]{FonUll2024DIE} system~\eqref{sys_modified} has at least $M+1$ geometrically distinct $T$-periodic solutions such that $y(0) \in\, \mathring{\mathcal{D}}\,$. By Proposition~\ref{pro:for:priori}, these solutions are indeed solutions of the original system~\eqref{sys:original}. \qed

\subsection{Proof of Theorem~\ref{Thm_res}}\label{Proof_2}
Similar to the previous section, we first state and prove the following proposition which is crucial for the proof of Theorem~\ref{Thm_res}.

\begin{pro}\label{pro:for:priori:res}
There exists $\ov{\rho}>1$ such that, for any $\rho>\ov{\rho}$, every $T$-periodic solution $(x(t),y(t),w(t))$ of~\eqref{sys_modified} satisfies $\|w\|_\infty \leq \ov{\rho}$.
\end{pro}
\begin{proof}
Assume by contradiction that there is a sequence $(\rho_n)_{n}$ in $]1,+\infty[$ and a sequence of $T$-periodic solutions $(x_n, y_n, w_n) $ of~\eqref{sys_modified}, with $\rho=\rho_n$, such that $\rho_n\to +\infty$ and $\|w_n\|_\infty > n$. Let $z_n=\frac{w_n}{\|w_n\|_\infty}$. Then by similar arguments in the proof of Proposition~\ref{pro:for:priori}, we can find a function $z \in C([0,T], \R^2)$ such that $\|z\|_{\infty}=1$, and  it solves
\begin{equation*}
    \begin{cases}
        J\dot{z}
        =(1-\Gamma(t))\nabla\mathscr{H}_1(z)+
\Gamma(t)\nabla\mathscr{H}_2(z)\,,\\
        z (0)= z (T)\,.
        \end{cases}
    \end{equation*}
By~\cite[Lemma 2.3]{FonGar2011JDE}, either $z$ is a solution of $J\dot{z}=\nabla\mathscr{H}_1 (z)$, or it is a solution of $J\dot{z}=\nabla\mathscr{H}_2 (z)$. Let us assume that $J\dot{z}=\nabla\mathscr{H}_1 (z)$, and so $\Gamma(t) \equiv 0$. The other case can be treated similarly. 

\medbreak

For suitable $R_0 >0$, $\theta_0 \in [0,\tau_1 [$\,, it will be $z(t)=R_0 \varphi(t+\theta_0)$.  By $A2$ and~\eqref{eq:of:lin:growth}, there exists $C_4>0$ such that
$$
|\dot{w}_n(t)| \leq C_4 (1+ |w_n(t)|)\,.
$$
So by  Gronwall Lemma we deduce that
$$
|w_n(t_0)| \leq |w_n(t_1)| e^{C_4(t_1-t_0)} \,,\quad \hbox{ for every } t_0 \leq t_1\,,
$$
and as a consequence we can say that if $\|w_n\|_{\infty} \to +\infty$, then $|w_n(t)| \to + \infty$ uniformly in $t \in [0,T]$. This shows that $w_n(t) \neq 0$ for every $t\in [0,T]$. Now writing, in generalized polar coordinates, $w_n (t)= r_n (t) \varphi(t+\theta_n (t))$, with $\theta_n (0) \in [0,\tau_1 [$ for every $n$, we see that $r_n(t) \to \infty$ uniformly in $t \in [0,T]$. The second system in~\eqref{sys_modified} with $\rho=\rho_n$ gives
\begin{align*}
    \dot{\theta}_n (t) &= \frac{\langle J\dot{w}_n (t)\,,w_n (t)\rangle}{r_n^2 (t)}-1\\
    &=  \frac{\Big\langle F_{\rho_n}\big(t,r_n (t) 
     \varphi(t+\theta_n (t))\big)\,,\varphi(t+\theta_n (t))\Big\rangle-r_n (t)}{r_n (t)}\\
    &\quad +\frac{\Big\langle\nabla_w P\big(t,x_n(t), y_n(t), r_n(t)\varphi(t+\theta_n (t))\big)\,,\varphi(t+\theta_n (t))\Big\rangle}{r_n (t)}\,.
\end{align*}
Since $z$ solves $J\dot{z}=\nabla\mathscr{H}_1 (z)$, by \eqref{res_condition} it performs $N$ rotations around the origin when $t$ varies from $0$ to $T$. Moreover, since $z_n \to z$ uniformly, for $n$ sufficiently large every $z_n$ performs $N$ rotations around the origin, and also every $w_n$, since $w_n = \|w_n\|_{\infty}z_n$. As a consequence, for such $n$ it is $\theta_n (0)=\theta_n (T)$, thus integrating in the above computation , we get
\begin{align*}
    0&= \int_0^T \Bigg[ \frac{\langle F_{\rho_n}\big(t,r_n (t) \varphi(t+\theta_n (t))\big)\,,\varphi(t+\theta_n (t))\rangle-r_n(t)}{r_n (t)}\\
     &\quad+\frac{\Big\langle\nabla_w P\big(t,x_n(t), y_n(t), r_n(t)\varphi(t+\theta_n (t))\big)\,,\varphi(t+\theta_n (t))\Big\rangle}{r_n (t)}\Bigg]\,dt\,.
     \end{align*}
So, for $n$ large, 
\begin{multline}\label{eq:to:use:Fatou}
   0= \int_0^T \Bigg[ \frac{\langle F_{\rho_{n}}\big(t,r_n (t) \varphi(t+\theta_n (t))\big)\,,\varphi(t+\theta_n (t))\rangle-r_n(t)}{r^v _n (t)}\\
     +\frac{\Big\langle\nabla_w P\big(t,x_n(t), y_n(t), r_n(t)\varphi(t+\theta_n (t))\big)\,,\varphi(t+\theta_n (t))\Big\rangle}{r^v _n (t)}\Bigg]\,dt\,,
\end{multline}
where $r^v _n (t)=\frac{r_n (t)}{\|w_n \|_\infty}$. 
We now prove that the expressions in the above computation are bounded below in order to apply Fatou's Lemma. Clearly, the second term in \eqref{eq:to:use:Fatou} is bounded below by $A2$. For the first one, we discuss the following cases.
\medbreak

\medbreak

\ni{\it Case 1.} If $|r_n (t) \varphi(t+\theta_n (t))| \leq \rho_n$, then $F_{\rho_n}=F$ and so by~\eqref{Euler_idty},~\eqref{covex_comb.},~\eqref{eq:of:bdd:of:q},~\eqref{H1_leq_H2}, and ~\eqref{half}, we have
$$
\big\langle F\big(t,r_n (t) \varphi(t+\theta_n (t))\big)\,,\varphi(t+\theta_n (t))\big\rangle-r_n(t) \geq  - \widetilde{C} \| \varphi\|_{\infty} \,.
$$

\medbreak

\ni{\it Case 2.} If $ |r_n (t) \varphi(t+\theta_n (t))| \geq \rho_n^3$, then 
\begin{multline*}
F_{\rho_n}(t, r_n (t) \varphi(t+\theta_n (t)))= \frac{r_n (t)}{2}\big[\nabla \mathscr{H}_1( \varphi(t+\theta_n (t)))+ \nabla\mathscr{H}_2( \varphi(t+\theta_n (t)))\big]\\
+ \nabla_w Q(t, r_n (t) \varphi(t+\theta_n (t)))\,.   
\end{multline*}
Using the fact that $\SH_1(w) < \SH_2(w)$, there exists $\varepsilon>0$ such that
$$
\SH_2(\varphi(s))- \SH_1(\varphi(s)) \geq \varepsilon\,, \quad \hbox{ for every } s \in [0,T]\,,
$$
and thus by~\eqref{Euler_idty},~\eqref{eq:of:bdd:of:q} and~\eqref{half}, we have
\begin{align}\label{eq:for:using:Fatou} \nonumber
    &\big\langle F_{\rho_n}\big(t,r_n (t)  \varphi(t+\theta_n (t))\big)\,,\varphi(t+\theta_n (t))\big\rangle-r_n(t)  \\ \nonumber
    & \geq \frac{r_n(t)}{2} \bigg( 1 + (1 + 2\varepsilon) \bigg) - r_n( t)+ \big\langle \nabla_w Q(t, r_n (t) \varphi(t+\theta_n (t))), \varphi(t+\theta_n (t)) \big\rangle\\
    & \geq \varepsilon r_n(t) - \widetilde{C}\| \varphi\|_{\infty} \to +\infty\,,
\end{align}
when $ n \to +\infty$, since $r_n(t) \to + \infty$.

\medbreak

\ni{\it Case 3.} Finally, if $\rho_n \leq |r_n (t) \varphi(t+\theta_n (t))|\leq \rho_n^3$, we just interpolate the previous two cases so to get the desired bound.

\medbreak

Thus we can use Fatou's Lemma to~\eqref{eq:to:use:Fatou} and get
\begin{multline*}
   0\geq  \int_0^T \liminf_n \Bigg[ \frac{\langle F_{\rho_{n}}\big(t,r_n (t) \varphi(t+\theta_n (t))\big)\,,\varphi(t+\theta_n (t))\rangle-r_n(t)}{r^v _n (t)}\\
     +\frac{\Big\langle\nabla_w P\big(t,x_n(t), y_n(t), r_n(t)\varphi(t+\theta_n (t))\big)\,,\varphi(t+\theta_n (t))\Big\rangle}{r^v _n (t)}\Bigg]\,dt\,.
\end{multline*}
Now using standard properties of the inferior limit, taking into account that, since $z_n \to z$ uniformly, $r^v _n \to R_0$ uniformly and by $A2$, we have
\begin{multline}\label{eq:for:Frho:and:F}
   \ov{m}\int_0 ^T |\varphi(t)|\,dt \\  \geq  \int_0^T \liminf_n  \big[\langle F_{\rho_{n}}\big(t,r_n (t) \varphi(t+\theta_n (t))\big)\,,\varphi(t+\theta_n (t))\rangle-r_n(t)\big]\,dt\,.
\end{multline}
We fix $t \in [0,T]$ and discuss the possible three cases to verify that
\begin{multline}\label{eq:of:Frho:and:F}
 \big\langle F_{\rho_{n}}\big(t,r_n (t) \varphi(t+\theta_n (t))\big)\,,\varphi(t+\theta_n (t))\big\rangle \\
 \geq \big\langle F\big(t,r_n (t) \varphi(t+\theta_n (t))\big)\,,\varphi(t+\theta_n (t))\big\rangle\,.
\end{multline}

\medbreak

\ni{\it Case 1.} If $|r_n (t) \varphi(t+\theta_n (t))| \leq \rho_n$, then $F_{\rho_n}=F$.

\medbreak

\ni{\it Case 2.} If $|r_n (t) \varphi(t+\theta_n (t))| \geq \rho_n^3$, then by \eqref{eq:for:using:Fatou} for $n$ large enough, \eqref{eq:of:Frho:and:F} follows.

\medbreak

\ni{\it Case 3.} Finally, if $\rho_n \leq |r_n (t) \varphi(t+\theta_n (t))|\leq \rho_n^3$, we just interpolate the previous two cases so to get \eqref{eq:of:Frho:and:F}, for $n$ large enough.
\medbreak
Thus \eqref{eq:for:Frho:and:F} implies that
$$
\ov{m}\int_0 ^T |\varphi(t)|\,dt \geq \int_0^T  \liminf_n  \big[\langle F\big(t,r_n (t) \varphi(t+\theta_n (t))\big)\,,\varphi(t+\theta_n (t))\rangle-r_n(t)\big]\,dt\,.
$$

\medbreak

By using the fact that $z_n \to z$ uniformly, we can assume without loss of generality that, up to a subsequence, $\theta_n (t) \to \theta_0 \in [0,T]$  uniformly. Recalling~\eqref{half}, for every $t\in [0,T]$ we are computing the inferior limit which appears in~\eqref{LL_1a} along the particular
subsequence $(r_n (t)\,, \theta_n (t))$ for which, $\theta_n (t) \to \theta_0$ and $r_n (t)=||w_n ||_\infty r^v _n (t)\to +\infty$. We obtain
\begin{equation*}
\ov{m}\int_{0}^{T}|\varphi(t)| dt\geq \int_{0}^{T}\liminf\limits_{(\lambda,\,s)\to(+\infty,\,\theta_0)}\big[\big\langle F\big(t,\lambda \varphi(t+s)\big),\,\varphi(t+s) \big\rangle -\lambda \big)\big]\,dt \,,
\end{equation*}
which contradicts~\eqref{LL_1a}, thus completing the proof of the proposition.
\end{proof}
To complete the proof of Theorem~\ref{Thm_res}, we can write exactly the same lines as for Theorem~\ref{Thm_nonres} after the proof of Proposition~\ref{pro:for:priori}. \qed

\subsection{Proof of Theorems~\ref{Thm_simple_res:below} and~\ref{Thm_simple_res:above}}
We first prove Theorem~\ref{Thm_simple_res:below}, and for this we assume without loss of generality that $\SH_1(w) < \SH_2(w)$ for all $w \in \R^2 \setminus \{0\}$. Because if this is not true, then we can define a function $\widetilde{\SH}_2$ by
$$
\widetilde{\SH}_2(w)= \SH_2(w)+ \epsilon |w|^2
$$
for some $\epsilon>0$ small enough so that  
$$
\SH_1(w) < \widetilde{\SH}_2(w).
$$
Denote by $\widetilde{\tau}_\psi$, the minimal period corresponding to system
$$
J \dot{w}= \nabla \widetilde{\SH}_2(w)\,.
$$
We see that for $\epsilon$ small enough, it still satisfies
$$
\frac{T}{N+1}< \widetilde{\tau}_\psi <  \tau_2  \leq \tau_1 = \frac{T}{N}\,.
$$
The proof of Theorem ~\ref{Thm_simple_res:below} now follows similarly as the proof of Theorem~\ref{Thm_res}.

\medbreak

To prove Theorem~\ref{Thm_simple_res:above}, we can define a function $\widetilde{\SH}_1$ by
$$
\widetilde{\SH}_1(w)= \SH_1(w)- \epsilon |w|^2
$$
for some $\epsilon>0$ small enough so that  
$$
\SH_2(w) > \widetilde{\SH}_1(w).
$$
Denote by $\widetilde{\tau}_\varphi$, the minimal period corresponding to system
$$
J \dot{w}= \nabla \widetilde{\SH}_1(w)\,.
$$
It is clear that  for $\epsilon$ small enough, it satisfies
$$
\frac{T}{N+1}=  \tau_2  \leq \tau_1 < \widetilde{\tau}_\varphi < \frac{T}{N}\,,
$$
and the proof follows similarly as the proof of Theorem~\ref{Thm_res}.

\section{Some variants of the main results}\label{variants}
We first discuss some variants of our results by changing assumption $A4$. Let us denote respectively the four quadrants of the plane $\R^2$ by $E_1, E_2, E_3, $ and $E_4$. Instead of $A4$, we assume the following assumption.

\noindent $A4'.$ There are five functions $\gamma_j: \R\times E_j \to [0,1]$ and $Q: \R\times \R^2 \to \R$ such that
\begin{equation}\label{covex_comb.variants}
F(t,w)=(1-\gamma_j(t,w))\nabla \mathscr{H}_1 (w)+\gamma_j(t,w)\nabla \mathscr{H}_2 (w) + \nabla_w Q(t,w)\,,
\end{equation}
for all $w \in E_j$, where all the functions are continuous, $T$-periodic in $t$, and $\nabla_w Q(t,w)$ is uniformly bounded, i.e., there exists $\widetilde{C}>0$ such that
\begin{equation}\label{eq:of:bdd:of:q:variants}
|\nabla_w Q(t,w)| \leq  \widetilde{C}\,, \quad \hbox{ for every }\, (t,w) \in [0,T] \times \R^2\,.
\end{equation}

\medbreak

We now state the variant of our main results.

\begin{Thm}\label{Thm:variant:A4:nonres}
If in the statement of Theorem~\ref{Thm_nonres} we replace assumption $A4$ by $A4'$, then the same conclusion holds.
\end{Thm}

\begin{proof}
    Similar to Section~\ref{Proofs}, we can find a function $\varPhi: \R \times \R^2 \to \R$ such that $F(t,w)= \nabla_{w} \varPhi(t,w)+ \nabla_{w} Q(t,w)$, and
\begin{equation}\label{nabla_phi:variant}
\nabla_w \varPhi (t,w)=(1-\gamma_j(t,w))\nabla \mathscr{H}_1 (w)+\gamma_j(t,w)\nabla \mathscr{H}_2 (w)\,, \quad \hbox{ for every } \, w \in E_j\,.
\end{equation}
Now we modify system~\eqref{sys:original} by the similar way we modified in Section~\ref{Proofs} with the only change in defining $\gamma_{\rho}$, i.e., we define $\gamma_{j_\rho}: \R\times E_j \to [0,1]$ instead of $\gamma_{\rho}: \R\times \R^2 \to [0,1]$ as
$$
\gamma_{j_\rho}(t,w)=\begin{cases}
        \gamma_j (t,w) &\quad \mbox{ if }|w|\leq \rho\,,\\
        \eta_\rho (|w|)\gamma_j (t,w) +\frac{1}{2}(1-\eta_\rho (|w|)) &\quad \mbox{ if }\rho\leq |w|\leq \rho^3\,,\\
        \frac{1}{2} &\quad \mbox{ if }|w|\geq \rho^3\,.
    \end{cases}
$$
We can now prove the a priori bound for the modified system \eqref{sys_modified} in this case following the same lines as in the proof of Proposition \ref{pro:for:priori} with the only difference in the $\Gamma_{n}(t)$, which we can define in this case as follows:
$$
\Gamma_n(t)=\gamma_{j_{\rho_{n}}}(t,\|w_n\|_\infty z_n(t))\,, \quad \hbox{ when }  z_n \in \mathring {E}_j\,.
$$
We notice that $\Gamma_n(t)$ is not defined if $z_n(t)$ belongs to an axis. Anyway, the set of those $t's$ has zero measure, since if $n$ is large enough, one has
$$
\langle \dot{z}_n(t), -J z_n(t) \rangle >0\,, \quad \hbox{ for every } t \in [0,T]\,,
$$
hence $z_n(t) $ crosses the axes transversally. 
\medbreak
After the a priori bound, we rewrite the Hamiltonian function of system \eqref{sys_modified} as in \eqref{eq:of:K:rho} and by a similar way, we obtain $M+1$ geometrically distinct $T$-periodic solutions $(x(t),y(t),w(t))$ of system \eqref{sys:original} such that $y(0) \in\, \mathring{\mathcal{D}}\,$. The proof of the theorem is thus completed.    
\end{proof}

We now state the case of double resonance. We skip the proof for briefness. 

\begin{Thm}\label{Thm:variant:A4:res}
If in the statement of Theorem~\ref{Thm_res} we replace assumption $A4$ by $A4'$, then the same conclusion holds.
\end{Thm}
The cases of simple resonance can be treated similarly.

\begin{re}
It is worth noticing that we could discuss a more general case in assumption $A4'$, i.e., instead of using the quadrants $E_j$, consider $\alpha_1 < \alpha_2 < \dots < \alpha_n= \alpha_1 + 2 \pi$, and define the planar sectors
$$
\widehat{E}_j= \{ r e^{i \theta}\; : \; \alpha_j \leq \theta \leq \alpha_{j+1}\,,\; r \geq 0  \}\,.
$$

\end{re}

\medbreak

We now consider some variants of our main results by modifying $A3$, and to this aim, we first recall some definitions. A closed convex bounded subset $\mathcal{D}$ of $\mathbb{R}^M$ having nonempty interior $\mathring {\mathcal{D}}$ is said to be a convex body of $\mathbb{R}^{M}$. If we assume that $\mathcal{D}$ has a smooth boundary, then we denote the unit outward normal at $\xi \in \partial \mathcal{D}$ by $\nu_{\mathcal{D}}(\xi)$. Moreover, we say that $\mathcal{D}$ is strongly convex if for any $k\in \partial \mathcal{D}$, the map $\mathcal{F}: \mathcal{D} \to\R$ defined by $\mathcal{F}(\xi)= \langle \xi-k , \nu_{\mathcal{D}}(k)\rangle$ has a unique maximum point at $\xi=k$.

\medbreak

First we state the following ``avoiding rays'' assumption.

\medbreak

\noindent $A3'$. There exists a convex body $\mathcal{D}$ of $\mathbb{R}^{M}$, having a smooth boundary, such that for $\sigma \in \{-1,1\}$ and for every $C^1$-function $\mathcal{W}:[0,T] \to \R^{2L}$, all the solutions $(x,y)$ of system~\eqref{eq:of:hamiltonian:unperterbed:for:higher}
starting with $y(0)\in{\cal D}$ are defined on $[0,T]$, and 
$$
y(0) \in \partial{\mathcal{D}}\quad\Rightarrow\quad  
x(T)-x(0) \notin \{ \sigma\lambda\, \nu_{\mathcal{D}}(y(0)):\lambda\ge0 \}\,. 
$$

\begin{Thm}
If in the statements of Theorem~\ref{Thm_nonres},~\ref{Thm_simple_res:below},~\ref{Thm_simple_res:above},~\ref{Thm_res}, we replace assumption $A3$ by $A3'$, the same conclusion holds.
\end{Thm}

\begin{proof}
The arguments are the same as in Section~\ref{Proofs}, with the only difference that instead of applying~\cite[Theorem 1.1]{FonUll2024DIE}, we apply~\cite[Theorem 3.1]{FonUll2024DIE}.
\end{proof}

\medbreak

Now we state the following ``indefinite twist'' assumption.

\medbreak

\ni $A3''$. There are a strongly convex body ${\cal D}$ of $\R^M$ having a smooth boundary and a symmetric regular $M \times M$ matrix $\mathbb{A}$ such that for every $C^1$-function $\mathcal{W}:[0,T] \to \R^{2L}$, all the solutions $(x,y)$ of system~\eqref{eq:of:hamiltonian:unperterbed:for:higher}
starting with $y(0)\in{\cal D}$ are defined on $[0,T]$, and 
$$
y(0) \in \partial{\mathcal{D}}\quad\Rightarrow\quad    
\langle x(T)-x(0)\,,\, \mathbb{A} \nu_{\mathcal{D}}(y(0)) \rangle > 0\,.
$$

\medbreak

\begin{Thm}
If in the statements of Theorem~\ref{Thm_nonres},~\ref{Thm_simple_res:below},~\ref{Thm_simple_res:above},~\ref{Thm_res}, we replace assumption $A3$ by $A3''$, the same conclusion holds.
\end{Thm}

\begin{proof}
Apply~\cite[Theorem 3.2]{FonUll2024DIE} instead of~\cite[Theorem 1.1]{FonUll2024DIE} in Section~\ref{Proofs}.
\end{proof}

\section{Applications to scalar second order equations}\label{examples}

As a direct consequence of Theorem~\ref{Thm_res}, we discuss the main results in \cite[Theorem 5, Theorem 14]{FonMamSfe2024Preprint}. We consider the following system
\begin{equation}\label{sys:conseq}
\begin{cases}
\dot{x}= \nabla_y \mathcal{H}(t,x,y)+\nabla_y P(t,x,y,u)\,,\\  
\dot{y}= -\nabla_x \mathcal{H}(t,x,y)-\nabla_xP(t,x,y,u)\,, \\
\ddot{u} + g(t,u)  =\partial_{u} P(t,x,u)\,,
\end{cases}
\end{equation}
where all the involved functions are continuous and $T$ periodic in the variable $t$.

\medbreak

\noindent $B4.$ There exist some positive constants $\mu_1, \mu_2, \nu_1, \nu_2, C_6, C_7$ such that
\begin{equation}\label{eq:of:bdd:of:g:1}
\begin{cases}
    g(t,u) \leq \nu_1 u + C_6 \quad \hbox{if } u \leq 0\,,\\
    g(t,u) \geq \mu_1 u - C_6 \quad \hbox{if } u \geq 0\,,
\end{cases}
\end{equation}
\begin{equation}\label{eq:of:bdd:of:g:2}
\begin{cases}
    g(t,u) \geq \nu_2 u - C_7 \quad \hbox{if } u \leq 0\,,\\
    g(t,u) \leq \mu_2 u + C_7 \quad \hbox{if } u \geq 0\,.
\end{cases}
\end{equation}
Moreover, we consider the following Landesman--Lazer conditions.

\medbreak

\noindent $B5.$ For a nontrivial $T$-periodic solution $\varphi$ of $\ddot{\varphi}+ \mu_1 \varphi^{+} -\nu_1 \varphi^{-} =0$, one has
\begin{multline}\label{LL_1ex}
\int_{ \{ \varphi<0\}}\liminf\limits_{u \to - \infty}\big[ \nu_1 u- g(t,u)\big] |\varphi(t)| dt\\ 
+ \int_{ \{ \varphi>0\}}\liminf\limits_{u \to + \infty}\big[  g(t,u)-\mu_1 u\big] \varphi(t) dt>\ov{m}\int_{0}^{T}|\varphi(t)| dt\,.
\end{multline}

\medbreak

\noindent $B6.$
For a nontrivial $T$-periodic solution $\psi$ of $\ddot{\psi}+ \mu_2 \psi^{+} -\nu_2 \psi^{-} =0$, one has
\begin{multline}\label{LL_2ex}
\int_{ \{ \psi<0\}}\liminf\limits_{u \to - \infty}\big[  g(t,u)-\nu_2 u\big] |\psi(t)| dt\\ 
+ \int_{ \{ \psi>0\}}\liminf\limits_{u \to + \infty}\big[  \mu_2 u -g(t,u)\big] \psi(t) dt>\ov{m}\int_{0}^{T}|\psi(t)| dt\,.
\end{multline}

 \medbreak
\begin{Thm}\label{Thm_res:conseq}
    Let $A1-A3$ and $B4-B6$ hold true, and assume that there exists a positive integer $N$ such that
    \begin{equation}\label{nonres_condition:ex}
\frac{T}{N+1} = \frac{\pi}{\sqrt{\mu_2}}+\frac{\pi}{\sqrt{\nu_2}} \leq \frac{\pi}{\sqrt{\mu_1}}+\frac{\pi}{\sqrt{\nu_1}} = \frac{T}{N}\,.    
    \end{equation}
    Then system~\eqref{sys:conseq} has at least $M+1$ geometrically distinct $T$-periodic solutions, with $y(0)\in \mathring{\mathcal{D}}$.
\end{Thm}

\begin{proof}
    It can be shown (see e.g. \cite[Lemma 1]{Fab1995JDE}) that using conditions~\eqref{eq:of:bdd:of:g:1} and~\eqref{eq:of:bdd:of:g:2}, there exists a bounded function $h(t,u)$ such that
\begin{equation}\label{eq:of:g:in:appl}
    g(t,u)= \zeta_1(t,u) u^{+} - \zeta_2(t,u)u^{-} +h(t,u)\,,
\end{equation}
where  $u^{+}= \max\{u,0\}$, $u^{-}= \max\{-u,0\}$, and 
\begin{equation}\label{eq:of:zetas}
    \mu_1 \leq \zeta_1(t,u) \leq \mu_2 \,, \qquad \nu_1 \leq \zeta_2(t,u) \leq \nu_2\,,
\end{equation}
for almost every $t \in [0,T]$, and every $u \in \R$.
System \eqref{sys:conseq} can be written as
\begin{equation}\label{sys:conseq:2}
\begin{cases}
\dot{x}= \nabla_y \mathcal{H}(t,x,y)+\nabla_y P(t,x,y,u)\,,\\  
\dot{y}= -\nabla_x \mathcal{H}(t,x,y)-\nabla_xP(t,x,y,u)\,, \\
\dot{u} = v\,, \quad
-\dot{v} =  g(t,u) - \partial{u} P(t,x,u)\,.
\end{cases}
\end{equation}

\medbreak

Now for $w=(u,v)$ with the aim of applying Theorem \ref{Thm:variant:A4:res}, define
$$
\SH_1(w) = \frac12 \bigg[ \mu_1 (u^{+})^2 + \nu_1 (u^{-})^2 + v^{2} \bigg]\,, \quad  \SH_2(w) = \frac12 \bigg[ \mu_2 (u^{+})^2 + \nu_2 (u^{-})^2 + v^{2} \bigg] 
$$
$$
Q(t,w)=\int_{0}^{u} h(t,s)ds, \quad F(t,w)=
\Bigg(\begin{matrix}
    g(t,w) \\ v
\end{matrix}\Bigg)\,.
$$
For $u > 0,$ $w \in E_1 \cup E_4$ and $u^{-}=0$. Thus by choosing 
$$
\gamma_1(t,w)= \gamma_4(t,w)= \frac{\zeta_1(t,u)- \mu_1}{\mu_2-\mu_1}\,,
$$
it is clear by~\eqref{eq:of:zetas} that 
$$
0 \leq \gamma_1(t,w)= \gamma_4(t,w) \leq 1\,.
$$
Similarly, for $u < 0$, $u^{+}=0$ and $w \in E_2 \cup E_3$. Thus by defining 
$$
\gamma_2(t,w)= \gamma_3(t,w)= \frac{\zeta_2(t,u)- \nu_1}{\nu_2-\nu_1}\,,
$$
we see by~\eqref{eq:of:zetas} that 
$$
0 \leq \gamma_1(t,w)= \gamma_4(t,w) \leq 1\,.
$$
Hence we can use Theorem~\ref{Thm:variant:A4:res}, and then system~\eqref{sys:conseq} has at least $M+1$ geometrically distinct $T$-periodic solutions $(x,y,u)$ such that $y(0)\in \mathring{\mathcal{D}}.$
\end{proof}

We remark here that the main result for the nonresonance case \cite[Theorem 2]{FonMamSfe2024Preprint}
can also be obtained by a similar way using Theorem~\ref{Thm:variant:A4:nonres}.

\medbreak

As an example, for $M=1$, we consider the coupling of a pendulum-like equation with an asymmetric oscillator, i.e.,
\begin{equation}\label{system:for:the:example}
\begin{cases}
\ddot{q}+A \sin q  =e(t) +\partial_{q} P(t,x,u)\,,\\
\ddot{u} + f(u) + h(t,u) =\partial_{u} P(t,x,u)\,,
\end{cases}
\end{equation}
where $h$ is bounded,
$$
f(u)= \Bigg( \frac{\mu_1 + \mu_2}{2} + \frac{\mu_2 - \mu_1}{2} \cos u \Bigg)u^{+} - \Bigg( \frac{\nu_1 + \nu_2}{2} + \frac{\nu_2 - \nu_1}{2} \sin(u^{3}) \Bigg) u^{-}\,
$$
and the constants $A,\mu_1, \mu_2,\nu_1, \nu_2$ are all positive with $\mu_1 \leq \mu_2$ and $\nu_1 \leq \nu_2$. Assume that $P(t,x,u)$ is $T$-periodic in $t$ and $2\pi$-periodic in $x$, and that it has a bounded gradient with respect to $(x,u)$. Setting $E(t)=\int_0^te(s)\,ds$, system~\eqref{system:for:the:example} is equivalent to
\begin{equation*}\label{system:for:the:example:2}
\begin{cases}
\dot{q}= p+E(t)\,, \quad
\dot{p} = -A \sin q  + \partial_{q} P(t,x,u)\,,\\
\dot{u} = v\,,\quad
\dot{v}= - f(u) - h(t,u) +\partial_{u} P(t,x,u)\,.
\end{cases}
\end{equation*}
Assuming $e(t)$ to be $T$-periodic with
$$
\int_0^Te(t)\,dt=0\,,
$$
the function $E(t)$ is $T$-periodic, as well.


It has been proved in \cite[Section 3]{FonUll2024DIE}, that the twist condition $A3$ holds true. 
\medbreak
On the other hand, if we define $\mathscr{H}_1$, $\SH_2$ by
$$
\SH_1(w) = \frac12 \bigg[ \mu_1 (u^{+})^2 + \nu_1 (u^{-})^2 + v^{2} \bigg]\,, \quad  \SH_2(w) = \frac12 \bigg[ \mu_2 (u^{+})^2 + \nu_2 (u^{-})^2 + v^{2} \bigg] \,,
$$
with $w=(u,v)$, then $\mathscr{H}_1$ and $\SH_2$ are positive, positively homogeneous of degree $2$, $\SH_1(w) < \SH_2(w)$ for all $w \ne 0$. Thus all the solutions of system $J\dot{w}=\nabla\mathscr{H}_1(w)$ and $J \dot{w} = \nabla\SH_2(w)$ are periodic with fixed period, respectively 
$$
\tau_1 = \frac{\pi}{\sqrt{\mu_1}}+\frac{\pi}{\sqrt{\nu_1}}\,, \quad \tau_2=  \frac{\pi}{\sqrt{\mu_2}}+\frac{\pi}{\sqrt{\nu_2}}\,.
$$
Moreover, we assume the following Landesman--Lazer conditions.

\medbreak

\noindent $B5'.$ For a nontrivial $T$-periodic solution $\varphi$ of $\ddot{\varphi}+ \mu_1 \varphi^{+} -\nu_1 \varphi^{-} =0$, one has
\begin{multline}\label{LL_1ex_2}
\int_{ \{ \varphi<0\}}\liminf\limits_{u \to - \infty}\big[ - h(t,u)\big] |\varphi(t)| dt\\ 
+ \int_{ \{ \varphi>0\}}\liminf\limits_{u \to + \infty}  h(t,u) \varphi(t) dt>\ov{m}\int_{0}^{T}|\varphi(t)| dt\,.
\end{multline}

\medbreak

\noindent $B6'.$
For a nontrivial $T$-periodic solution $\psi$ of $\ddot{\psi}+ \mu_2 \psi^{+} -\nu_2 \psi^{-} =0$, one has
\begin{multline}\label{LL_2ex_2}
\int_{ \{ \psi<0\}}\liminf\limits_{u \to - \infty} h(t,u) |\psi(t)| dt\\ 
+ \int_{ \{ \psi>0\}}\liminf\limits_{u \to + \infty}\big[   -h(t,u)\big] \psi(t) dt>\ov{m}\int_{0}^{T}|\psi(t)| dt\,.
\end{multline}

 \medbreak

All the assumptions of Theorem~\ref{Thm_res:conseq} are satisfied, and we can thus state the following.

\begin{cor}\label{cor1}
In the above setting, assume that $B5', B6'$ hold, and 
$$
\frac{T}{N+1} = \frac{\pi}{\sqrt{\mu_2}}+\frac{\pi}{\sqrt{\nu_2}} \leq \frac{\pi}{\sqrt{\mu_1}}+\frac{\pi}{\sqrt{\nu_1}} = \frac{T}{N}\,.
$$
Then system~\eqref{system:for:the:example} has at least two geometrically distinct $T$-periodic solutions. 
\end{cor}

The above result generalizes a classical result by Mawhin and Willem~\cite{MawWil1984JDE} on the multiplicity of periodic solutions for the pendulum equation.

\section{Neumann-type boundary value problem}\label{NBVP}

Consider system \eqref{sys:original} with Neumann-type boundary conditions
\begin{equation}\label{eq:of:nbc}
    \begin{cases}
        y(a)=0=y(b)\,,\\
        v(a)=0=v(b)\,,
    \end{cases}
\end{equation}
with $w=(u,v)$. Notice that in this case, we omit the periodicity in the variable $t$ of all the functions in system~\eqref{sys:original}. We only consider all the involved functions to be continuous and continuously differentiable with respect to $(x,y,w)$. For more details in this direction, we refer the reader to some recent advancements \cite{BosGar2016JAA, FonMamObeSfe2024NoDEA, FonOrt2023RCMP, MarSfe2023ADE, Sfe2020TMNA}.
\medbreak

For the Neumann-type boundary value problem, instead of assumptions $A3$, $A4$, $A5$ and $A6$, we state the following assumptions.

\medbreak
\ni $A3'''$. All the solutions of system~\eqref{sys:original} satisfying $y(a)=v(a)=0$ are defined on $[a,b]$. 

\medbreak

\noindent $A4'''.$ There are two functions $\gamma: [a,b]\times \R^2 \to [0,1]$ and $Q: [a,b]\times \R^2 \to \R$ such that
\begin{equation}\label{covex_comb.nbp}
F(t,w)=(1-\gamma(t,w))\nabla \mathscr{H}_1 (w)+\gamma(t,w)\nabla \mathscr{H}_2 (w) + \nabla_w Q(t,w)\,,
\end{equation}
where all the functions are continuous and $\nabla_w Q(t,w)$ is uniformly bounded, i.e., there exists $\widetilde{C}>0$ such that
\begin{equation}\label{eq:of:bdd:of:q:nbp}
|\nabla_w Q(t,w)| \leq  \widetilde{C}\,, \quad \hbox{ for every }\, (t,w) \in [a,b] \times \R^2\,.
\end{equation}

\medbreak

\noindent $A5'''.$ For every $\theta  \in[a,b]$, one has
\begin{equation}\label{LL_1a:nbp}
\int_{a}^{b}\liminf\limits_{(\lambda,\,s)\to(+\infty,\,\theta)}\big[\big\langle F\big(t,\lambda \varphi(t+s)\big),\,\varphi(t+s) \big\rangle -2\lambda \mathscr{H}_1\big(\varphi(t)\big)\big] \,dt>\ov{m}\int_{a}^{b}|\varphi(t)| dt\,.
\end{equation}

\medbreak

\noindent $A6'''.$
For every $\theta  \in[a,b]$, one has
\begin{equation}\label{LL_1b:nbp}
\int_{a}^{b}\liminf\limits_{(\lambda,\,s)\to(+\infty,\,\theta)}\big[2\lambda \mathscr{H}_2\big(\psi(t)\big)-\big\langle F\big(t,\lambda \psi(t+s)\big),\,\psi(t+s) \big\rangle \big] \,dt>\ov{m}\int_{a}^{b}|\psi(t)| dt\,.
\end{equation}

\medbreak

Considering problem \eqref{sys:original}--\eqref{eq:of:nbc}, if $u_0 <0$, for all solutions $\zeta$ of
\begin{equation}\label{system_z:higher1}
 \dot{u}= \nabla_v \mathscr{H}_j(u,v)\,,\qquad
 \dot{v} = - \nabla_u \mathscr{H}_j(u,v)\,,
\end{equation}
starting with $\zeta(0)=(u_{0},0)$, there is a first time $\tau_{j_+}>0$ for which $v(\tau_{j_+})=0$, while $v(t)>0$ for all $t \in\, ]0, \tau_{j_+}[$\,, and this time $\tau_{j_+}$ is independent of $u_0 <0$. Similarly, if $u_0 >0$, there is a first time $\tau_{j_-}>0$ for which $v(\tau_{j_-})=0$, while $v(t)<0$ for all $t \in\, ]0, \tau_{j_-}[$\,, and this time $\tau_{j_-}$ is independent of $u_0 >0$. Clearly enough, $\tau_j=\tau_{j_+}+ \tau_{j_-}$ for $j=1,2$.

Here is our first main result which is associated with a nonresonant situation.
\begin{Thm}\label{Thm_nonres:nbp}
    Let $A1-A2$ and $A3'''-A4'''$ hold true and for $j=1,2$, $\tau_{j_{+}}= \tau_{j_{-}}$. Assume moreover that there exists a positive integer $N$ such that
    \begin{equation}\label{nonres_condition:nbp}
\frac{b-a}{N+1}< \tau_{2_+} \leq \tau_{1_+} < \frac{b-a}{N}\,.
    \end{equation}
    Then problem~\eqref{sys:original}--~\eqref{eq:of:nbc} has at least $M+1$ geometrically distinct solutions.
\end{Thm}

\medbreak
Concerning the double resonance case, we have the following result. 

\begin{Thm}\label{Thm_res:nbp}
    Let $A1-A2$ and $A3'''-A6'''$ hold true, for $j=1,2$, $\tau_{j_{+}}= \tau_{j_{-}}$ and $\SH_1(w) < \SH_2(w)$ for all $w \in \R^2 \setminus \{0\}$. Assume moreover that there exists a positive integer $N$ such that
    \begin{equation}\label{res_condition:nbp}
\tau_{2_+}=\frac{b-a}{N+1} \quad \hbox{ and } \quad \tau_{1_+} = \frac{b-a}{N}\,.
    \end{equation}  
    Then the same conclusion of Theorem~\ref{Thm_nonres:nbp} holds.
\end{Thm}

Let us remark here that a sufficient condition for having satisfied the assumption $\tau_{j_{+}}=\tau_{j_{-}}$ is that the function $\mathscr{H}_j$ is even in $v$. This is a frequent case in the applications, where, e.g., $\mathscr{H}_j$ is quadratic in $v$.

\medbreak

We only prove the double resonance result. In order to prove this theorem, we require the following result for the Neumann-type problem which is similar to~\cite[Lemma 2.3]{FonGar2011JDE} (for the periodic problem). We omit the proof of nonresonance case and statements of simple resonance cases for briefness.

\begin{lem}\label{Lemma:nbp}
    Assume that \eqref{res_condition:nbp} is satisfied. Then, if $v\in H^1 (a,b)$ solves 
    \begin{equation}\label{eq:Lemma}
        \begin{cases}
            J\dot{v}=\alpha(t)\nabla H_1 (v)+ (1-\alpha(t))\nabla H_2 (v)\,,\\
            v(a)=0=v(b)\,,
        \end{cases}
    \end{equation}
    being $\alpha \in L^2 (a,b)$, with $0\leq \alpha(t)\leq 1$ for almost every $t\in [0,T]$, then $v$ solves either
    $$
    J\dot{v}=\nabla H_1 (v)\,,
    $$
    or 
    $$
    J\dot{v}=\nabla H_2 (v)\,.
    $$
\end{lem}

\begin{proof}
By a remark after \cite[Lemma 2.2]{FonGar2011JDE}, we observe that a nontrivial solution of \eqref{eq:Lemma} never reaches the origin. Indeed, if $v(t)$ solves \eqref{eq:Lemma} then also $s_0 v(t)$ does, for every $s_0>0$, thanks to the homogeneity of the right-hand side; moreover, since the right-hand side grows at most linearly in $v$, \cite[Lemma 2.2]{FonGar2011JDE} applies and also the remark after that lemma. It follows that, $v(\bar{t})\neq 0$ for some $\bar{t}\in [a,b]$, then $v(t)\neq 0$ for every $t\in [a,b]$.
\medbreak
Consequently, the usual system in polar coordinates $(\rho,\theta)$ is well defined for system \eqref{eq:Lemma}. Writing $v(t)=\rho(t)(\cos\theta(t),\sin\theta(t))$, recalling Euler's formula we get
$$
-\dot{\theta}(t)=2\alpha(t)H_1\big(\cos\theta(t),\sin\theta(t)\big)+2(1-\alpha(t))H_2\big(\cos\theta(t),\sin\theta(t)\big)\,.$$
Since $0\leq \alpha(t)\leq 1$, it follows that
\begin{equation}\label{theta:dot}
    -\frac{\dot{\theta}(t)}{2H_2(\cos\theta(t),\sin\theta(t))}\leq 1\leq -\frac{\dot{\theta}(t)}{2H_1(\cos\theta(t),\sin\theta(t))}\,,
\end{equation}
for almost every $t\in [a,b]$. Since $v(a)=0=v(b)$, $v$ performs an integer number of half rotations around the origin, say $m$ half rotations. Recalling that
$$
\int_0 ^\pi \frac{d\theta}{2H_1(\cos\theta,\sin\theta)}=\tau_{1_+}\,,\quad \int_0 ^\pi \frac{d\theta}{2H_2(\cos\theta,\sin\theta)}=\tau_{2_+}\,,
$$
integrating in \eqref{theta:dot} from $a$ to $b$, we get
$$
m\tau_{2_+} \leq b-a \leq m\tau_{1_+}\,,
$$
from which, using \eqref{res_condition:nbp},
$$
N= \frac{b-a}{\tau_{1_+}}\leq m \leq \frac{b-a}{\tau_{2_+}}= N+1\,.
$$
Since $m$ is integer, this gives a contradiction unless $m=N$ and $\tau_{1_+}=\frac{b-a}{N}$ or $m=N+1$ and $\tau_{2_+}=\frac{b-a}{N+1}$.
\medbreak
We only consider the first case, the other can be treated in a similar way. Introducing the generalized polar coordinates in \eqref{eq:Lemma} by writing $v(t)=r(t)\psi(t+\omega(t))$, we get
\begin{equation}\label{r:dot}
\dot{r}(t)=-r(t)\alpha(t)\big\langle \nabla H_1 \big(\psi(t+\omega(t))\big),\dot{\psi}(t+\omega(t))\big\rangle\,,
    \end{equation}
    and 
    \begin{equation}\label{omega:dot}
       \dot{\omega}(t) = \alpha(t)(2H_1(\psi(t+\omega(t))-1).
    \end{equation}
Since $\omega(a)=0=\omega(b)$, integrating in \eqref{omega:dot} from $a$ to $b$, we have
$$
0=\int_a^b \alpha(t)(2H_1(\psi(t+\omega(t))-1) dt\,,
$$
and from the fact that $-\alpha(t)(2H_1(\psi(t+\omega(t))-1)\geq 0$ for almost every $t\in [a,b]$, this expression satisfies
\begin{equation}\label{condition:equal:to:0}
    \alpha(t)(2H_1(\psi(t+\omega(t))-1)=0
\end{equation}
almost everywhere, that is to say, $\dot{\omega}(t)=0$ for almost every $t\in [a,b]$. Thus, since $\omega(t)$ is absolutely continuous, there exists $\omega_0 \in \R$ such that $\omega(t)=\omega_0$ for every $t\in [a,b]$. Concerning \eqref{r:dot}, we have
$$
\dot{r}(t)=-r(t)\alpha(t)\big\langle \nabla H_1 \big(\psi(t+\omega(t))\big),\dot{\psi}(t+\omega(t))\big\rangle\,.
$$
We now need to prove that $\dot{r}(t)=0$ for almost every $t\in [a,b]$, Indeed, if $t\in [a,b]$, \eqref{condition:equal:to:0} implies that either $\alpha(t)=0$, or $H_1 \big(\psi(t+\omega_0\big)=\frac{1}{2}$. If $\alpha(\bar{t})=0$, then $\dot{r}(\bar{t})=0$; on the other hand, if $\alpha(\bar{t})>0$, then $\bar{t}$ is a zero of the function $t\mapsto H_1 \big(\psi(t+\omega_0\big)-H_2 \big(\psi(t+\omega_0\big)$, which is of class $C^1$ and nonnegative. Necessarily $\bar{t}$ is then a minimum of this function, and so
$$
\frac{d}{dt}H_1 \big(\psi(t+\omega_0\big)_{|_{t=\bar{t}}}=\frac{d}{dt}H_2 \big(\psi(t+\omega_0\big)_{|_{t=\bar{t}}}=0\,,
$$
as $H_2$ is preserved along $\psi$. It follows that
$\big\langle \nabla H_1 \big(\psi(\bar{t}+\omega_0\big),\dot{\psi}(\bar{t}+\omega_0\big\rangle=0$, so that $\dot{r}(\bar{t})=0$. Summing up, $\dot{r}(t)=0$ for almost every $t\in [a,b]$, and, since $r(t)$ is absolutely continuous, this implies that $r(t)$ is constant; being $v(t)=R_0\varphi(t+\omega_0)$ for some nonnegative constant $R_0$, it follows that $v$ is a solution of $J\dot{v}=\nabla H_2(v)$\,. This concludes the proof.
\end{proof}
To conclude the proof of Theorem \ref{Thm_res:nbp}, one can follow the same lines of the proof of Theorem \ref{Thm_res}. We need to use Lemma \ref{Lemma:nbp} instead of \cite[Lemma 2.3]{FonGar2011JDE} in Proposition \ref{pro:for:priori:res} and at the end, we have to use ~\cite[Theorem 1.1]{FonUll2024NODEA} (for $p=q=2$) instead of ~\cite[Theorem 1.1]{FonUll2024DIE}.\qed

\medbreak

We can discuss variants of the above results as like Theorem \ref{Thm:variant:A4:nonres} and Theorem \ref{Thm:variant:A4:res}. We are not entering into detail in order to avoid repetition of arguments. Furthermore, we can adapt an example of application for this situation as in Section \ref{examples}.

\section{Open problems}\label{open} 
We conclude the paper by suggesting some open problems.

\medbreak

\noindent{\sl 1.} In \cite{CheQia2022JDE}, the coupling of twist dynamics with a resonant equation involving the Ahmad–Lazer–Paul condition has been addressed. However, under such
a condition, the scenario of double resonance remains unexplored, even when
dealing only with a scalar second order equation. Notice that the Ahmad–
Lazer–Paul condition does not guarantee an a priori bound as the one proved
in Proposition \ref{pro:for:priori}, as shown in \cite{BosGar2019RIMUT}. A similar result for Neumann-type boundary condition is still open for consideration.

\medbreak
 
\noindent{\sl 2.} In \cite{FonUll2024NODEA}, the authors studied about a Hamiltonian system coupled with positive and positively $(p,q)$-homogeneous Hamiltonian system. They have discussed both periodic and Neumann-type problems for the nonresonance case. It would be interesting to try the nonresonance/resonance cases by taking positive and positively $(p,q)$-homogeneous Hamiltonian functions in \eqref{covex_comb.}.

\noindent{\sl 3.} In \cite{FonMamSfeUll2024pre}, they investigated a system ruled by a  positively-$(p,q)$-homoge\-ne\-ous Hamiltonian function with a friction term and used a nonresonance condition developed by  Frederickson--Lazer \cite{FreLaz1969JDE}. We wonder whether in our setting, Frederickson--Lazer-condition could be applied.

\bigbreak

\ni {\bf Acknowledgement}. We are deeply grateful to Prof. Alessandro Fonda for generously allowing us to explore this research problem combined with his insightful guidance  and valuable discussions. We extend our sincere gratitude to Prof. Andrea Sfecci for his valuable comments.

 \newpage

\noindent Authors' addresses:

\bigbreak

\begin{tabular}{l}
Natnael Gezahegn Mamo and Wahid Ullah\\
Dipartimento di Matematica, Informatica e Geoscienze\\
Universit\`a degli Studi di Trieste\\
P.le Europa 1, 34127 Trieste, Italy\vspace{1mm}\\
e-mail: natnaelgezahegn.mamo@phd.units.it, 
 wahid.ullah@phd.units.it
\end{tabular}

\bigbreak

\noindent Mathematics Subject Classification:  34B15; 34C25

\medbreak

\noindent Keywords: Hamiltonian systems; periodic solutions; positively 2-homo\-geneous systems; Landesman-Lazer condition; Poincar\'e--Birkhoff Theorem; Resonance; Neumann-type boundary conditions.


\begin{thebibliography}{999}

\bibitem{BosGar2016JAA} A. Boscaggin and M. Garrione, Resonant Sturm–Liouville boundary value problems for differential systems in the plane, J. Anal. Appl. 35 (2016), 41--59.


\bibitem{BosGar2019RIMUT} A. Boscaggin and M. Garrione, A counterexample to a priori bounds under the Ahmad--Lazer--Paul condition, Rend. Istit. Mat. Univ. Trieste 51 (2019), 33--39.

\bibitem{BreNir1978ASNSP}
H. Br\'ezis and L. Nirenberg, Characterizations of the ranges of some nonlinear operators and applications to boundary value problems, Ann. Scuola Norm. Sup. Pisa 5(2) (1978), 225--326.



\bibitem{CheQia2022JDE}
F. Chen and D. Qian, An extension of the Poincar\'e--Birkhoff theorem for Hamiltonian systems coupling resonant linear components with
twisting components, J. Differential Equations 321 (2022), 415--448.

\bibitem{Dra1990JDE} 
P. Dr\'abek, Landesman--Lazer condition for nonlinear problems with jumping nonlinearities, J. Differential Equations 85 (1990), 186--199.

\bibitem{FabFon1992NA}
C. Fabry and A. Fonda, Nonlinear equations at resonance and generalized eigenvalue problems, Nonlinear Anal. 18 (1992) 427--444.

\bibitem{Fab1995JDE} C. Fabry, Landesman-Lazer conditions for periodic boundary value problems with asymmetric nonlinearities, J. Differential Equations 116(2) (1995), 405-418.

\bibitem{Fon2004JDE}
A. Fonda, Positively homogeneous Hamiltonian systems in the plane, J. Differential equations 200 (2004) 162-184.

\bibitem{FonGar2011JDE}
A. Fonda and M. Garrione, Double resonance with Landesman--Lazer conditions for planar systems of ordinary differential equations, J. Differential Equations 250 (2011) 1052--1082.


\bibitem{FonGarSfe2023JMAA} 
A. Fonda, M. Garz\'on and A. Sfecci, An extension of the Poincar\'e--Birkhoff Theorem coupling twist with lower and upper solutions, J. Math. Anal. Appl. 528 (2023), Paper No. 127599, 33 pp.

\bibitem{FonGid2020NODEA} A. Fonda and P. Gidoni, Coupling linearity and twist: an extension of the Poincar\'e--Birkhoff theorem for Hamiltonian systems, NoDEA Nonlinear Differential Equations Appl. 27 (2020), Paper No. 55, 26 pp.

\bibitem{FonMamObeSfe2024NoDEA}
A. Fonda, N.G. Mamo, F. Obersnel and A. Sfecci, Multiplicity results for Hamiltonian systems with Neumann-type boundary conditions,
  NoDEA Nonlinear Differential Equations Appl. 31 (2024), Paper No. 31, 30 pp.


\bibitem{FonMamSfe2024Preprint}
 A. Fonda, N.G. Mamo and A. Sfecci,
An extension of the Poincaré–Birkhoff Theorem to systems involving Landesman–Lazer conditions, Ricerche Mat., to appear.

\bibitem{FonMamSfeUll2024pre} A. Fonda, N.G. Mamo, A. Sfecci and W. Ullah, Perturbed positively-$(p,q)$-homogeneous Hamiltonian systems with Frederickson--Lazer conditions, Preprint 2024.

\bibitem{FonOrt2023RCMP}  A. Fonda and R. Ortega, A two-point boundary value problem associated with Hamiltonian systems on a cylinder, Rendiconti del Circolo Matematico di Palermo 72 (2023), 3931-3947.


\bibitem{FonSfeToa2024Pre}
A. Fonda, A. Sfecci and R. Toader, Multiplicity of periodic solutions for nearly resonant Hamiltonian systems, Preprint 2024.



\bibitem{FonUll2024JDE}
A. Fonda and W. Ullah, Periodic solutions of Hamiltonian systems coupling twist with generalized lower/upper solutions, J. Differential Equations 379 (2024), 148--174.

\bibitem{FonUll2024DIE} A. Fonda and W. Ullah, Periodic solutions of Hamiltonian systems coupling twist with an isochronous center, Differential Integral Equations 37 (2024), 323--336.

\bibitem{FonUll2024NODEA} A. Fonda and W. Ullah,
Boundary value problems associated with Hamiltonian systems coupled with positively $(p,q)$ homogeneous systems, NoDEA Nonlinear Differential Equations Appl. 31 (2024), Paper No. 41, 28 pp.

\bibitem{FonUre2017AIHPANL}
A. Fonda and A.J. Ure\~na, A higher dimensional Poincar\'e--Birkhoff theorem for Hamiltonian flows, Ann. Inst. H. Poincar\'e Anal. Non Lin\'eaire 34 (2017), 679--698.

 \bibitem{FreLaz1969JDE} P.O. Frederickson, and A.C. Lazer, Necessary and sufficient damping in a second-order oscillator, J. Differential Equations 5 ( 1969), 262--270.          

\bibitem{LanLaz1970JMM}
{E.M. Landesman and A.C. Lazer},
{Nonlinear perturbations of linear elliptic boundary value problems at resonance},
J. Math. Mech. 19 (1970), 609--623.

\bibitem{LazLea1969AMPA} A.C. Lazer and D.E. Leach, Bounded perturbations of forced harmonic oscillators at resonance, Ann. Mat. Pura Appl. 82 (1969), 49--68.

\bibitem{MarSfe2023ADE} R. Marvulli and A. Sfecci, Landesman–Lazer type conditions for scalar one-sided superlinear nonlinearities with Neumann boundary conditions,
Adv. Differential Equations 28 (2023), 247--286.



\bibitem{Maw1977CSMUB}
J. Mawhin, Landesman--Lazer's type problems for nonlinear equations, Confer. Sem. Mat. Univ. Bari 147 (1977), 22 pp.

\bibitem{MawWil1984JDE} J. Mawhin and M. Willem, Multiple solutions of the periodic boundary value problem for some forced pendulum-type equations, J. Differential Equations 52 (1984), 264--287.

\bibitem{Sfe2020TMNA}
A. Sfecci, Double resonance in Sturm--Liouville planar boundary value problems, Topol. Methods Nonlinear Anal. 55 (2020), 655--680.

\end{thebibliography}
\end{document}